\documentclass{amsart}

\usepackage{url}
\usepackage{epigraph}
\usepackage[backref=page,colorlinks,bookmarksopen=true]{hyperref}
\usepackage{amsfonts,amsmath,amssymb}
\usepackage{amsthm}
\usepackage{enumerate}
\usepackage{appendix}
\usepackage{pgf,tikz}
\usetikzlibrary{arrows}

\usepackage{mathrsfs, xspace}
\usepackage{stmaryrd}
\usepackage{graphicx,psfrag}
\usepackage{a4wide}

\usepackage[all]{xy}
\theoremstyle{plain}
\newtheorem{proposition}{Proposition}[section]
\newtheorem{theorem}[proposition]{Theorem}

\newtheorem{lemma}[proposition]{Lemma}
\newtheorem{corollary}[proposition]{Corollary}

\theoremstyle{definition}
\newtheorem{definition}[proposition]{Definition}

\newtheorem*{ack}{Acknowledgements}

\theoremstyle{remark}
\newtheorem{remark}[proposition]{Remark}
\newtheorem{remarks}[proposition]{Remarks}

\everymath{\displaystyle}

\DeclareMathOperator{\Stab}{Stab}

\DeclareMathOperator{\Fix}{Fix}
\DeclareMathOperator{\Prob}{Prob}

\newcommand{\ms}[1]{\mathscr{#1}} 
\newcommand{\mc}[1]{\mathcal{#1}}

\newcommand{\ol}[1]{\overline{#1}}

\newcommand{\action}{\curvearrowright}
\newcommand{\barFH}{(\overline{\textrm{FH}})}

\newcommand{\CR}{\mathcal{C}}

\newcommand{\Conv}{\mathcal{C}}

\newcommand{\R}{\mathbb{R}}                          
\newcommand{\N}{\mathbb{N}}                          
\newcommand{\Z}{\mathbb{Z}}                          

\DeclareMathOperator{\dd}{d\!}

\newcommand{\SG}{\mathcal{S}(G)}
\newcommand{\RG}{\mathrm{R}_{\mathrm{a}}(G)}
\newcommand{\ASG}{\mathcal{S}_{\mathrm{a}}(G)}
\newcommand{\RASG}{\mathcal{S}_{\mathrm{ra}}(G)}

\newcommand{\IRS}{\mathrm{IRS}(G)}

 \makeatletter
\let\@wraptoccontribs\wraptoccontribs
\makeatother
 
\begin{document}
\title{Amenable Invariant Random Subgroups}

\author{Uri Bader}
\address{Mathematics Department, Technion - Israel Institute of Technology, Haifa, 32000, Israel}
\email{uri.bader@gmail.com}

\author{Bruno Duchesne}
\address{Institut \'Elie Cartan de Lorraine, Universit\'e de Lorraine, B.P. 70239, F-54506 Vandoeuvre-l\`es-Nancy Cedex, France}
\email{bruno.duchesne@univ-lorraine.fr}

\author{Jean L\'ecureux}
\address{D\'epartement de Math\'ematiques - B\^atiment 425, Facult\'e des Sciences d'Orsay, Universit\'e Paris-Sud 11, F-91405 Orsay, France}
\email{jean.lecureux@math.u-psud.fr}

\contrib[With an appendix by]{Phillip Wesolek}
\address{Universit\'e catholique de Louvain, Institut de Recherche en Math\'ematiques et Physique (IRMP), Chemin du Cyclotron 2, box L7.01.02, 1348 Louvain-la-Neuve, Belgium}
\email{phillip.wesolek@uclouvain.be}

\date{\today}
\thanks{U.B acknowledges the support of the European Research Council. B.D. is supported in part by Lorraine Region and Lorraine University.}

\begin{abstract} We show that an amenable Invariant Random Subgroup of a locally compact second countable group lives in the amenable radical. This answers a question raised in the introduction of \cite{Abert:2012jk}.
We also consider an opposite direction, property (T), and prove a similar statement for this property.
\end{abstract}
\maketitle

%
%

\epigraph{\it Yes, the IRS is amenable to working with you if you are cooperative and willing to work with them.}{--- taxes.answers.com}
\section{Introduction}

\subsection{Invariant random subgroups}
Let $G$ be a locally compact second countable group.
We denote by $\mathcal{S}(G)$ the space of closed subgroups of $G$. We endow it with the Chabauty topology which is defined in the following way. For $K\subseteq G$ compact and $U\subseteq G$ open, we define the two following subsets of $\SG$
\[ O^K=\{H\in\SG;\ H\cap K\neq\emptyset\}\ \mathrm{and}\ O_U=\{H\in\SG;\ H\cap U=\emptyset\}.\]
The Chabauty topology is then the smallest topology on $\SG$ containing the $O_K$'s and $O_U$'s when $K$ and $U$ vary respectively among compact and open subspaces of $G$. With this topology $\SG$ is a compact metrizable space \cite[Propositions 1.7\&1.8]{MR2406240}. Maybe more concretely, a sequence $(H_n)$ of closed subgroups converges to $H$ if and only if the following properties hold.

\begin{itemize}
\item For any $h\in H$ and $n\in\N$, there is $h_n\in H_n$ such that $(h_n)$ converges to  $h$.
\item For any converging sequence $(h_n)$ such that  $h_n\in H_n$ for all $n\in \N$, the limit is in $H$.
\end{itemize}

Recently, a new and fruitful point of view about --- non-free --- probability measure preserving (shortly p.m.p) actions appeared and is currently a fast growing field of research.

\begin{definition}An \emph{invariant random subgroup} (shortly IRS) is a Borel probability measure on $\SG$ which is invariant under the adjoint action of $G$ on $\SG$ by conjugations.
\end{definition}

We denote the space of IRSs by $\IRS$. A probabilistic point of view is the following: an IRS is a random closed subgroup of $G$ whose distribution is invariant under conjugations. We will alternate between the two points of view depending on the desire of short statements or precise ones.

The name \emph{IRS} first appeared in \cite{Abert:2012jk}  for countable groups. We refer to \cite{Abert:2012jk} for historical background on IRSs before the name was coined. Since this first appearance, IRSs appeared in several papers, for example in \cite{Abert:2012qy} and \cite{Biringer:2014fr} where recent references are given in the introduction.

Standard examples of IRSs are given by closed normal subgroups (Dirac masses in $\SG$) and lattices (in that case the measure is supported on the conjugacy class of the lattice). Thus, an IRS may be thought as a generalization of both normal subgroups and lattices. A general idea is that a statement true for normal subgroups and lattices should be true for IRSs.

There is a more general way to construct IRSs: Let $G\action X$ be a  p.m.p action, the map
\[\begin{array}{ccc}
  X & \to & \SG \\
  x & \mapsto & \Stab_G(x) \\
\end{array}\]
allows us to push the invariant probability measure on $X$ to $\SG$ and obtain an IRS. Actually, any IRS can be obtained that way (see \cite[Proposition 14]{Abert:2012jk} in the discrete case and  \cite{Abert:2012qy} in the general case). So IRSs can be seen as p.m.p actions outside the classical world of free actions (where $\Stab_G(x)=\{e\}$ for almost all $x\in X$). An IRS $\mu$ is said to be \emph{ergodic} if $G\action(\SG,\mu)$ is ergodic.

\subsection{Amenability}

In their generalization of Kesten theorem about amenable normal subgroups to IRSs, the authors of \cite{Abert:2012jk} were led to the study of amenable IRSs, which are IRSs supported on amenable subgroups. We denote by $\ASG$ the subspace of $\SG$ consisting of amenable closed subgroups. It is an open question to decide whether $\ASG$ is closed. Nevertheless, P. Wesolek shows in Appendix \ref{appendix} that it is a Borel subset. Recall that among amenable closed normal subgroups, there is a unique maximal one, which is called the \emph{amenable radical} of $G$ (see for example \cite[Proposition 4.1.12]{MR776417}). We denote it by $\RG$.

Since we consider only subgroups of $G$, it is more natural to use the  notion of \emph{relative} amenability as introduced in \cite{Caprace:2013kq} via the fixed point property. Let us recall that $G$ is amenable if any non-empty convex compact $G$-space has a fixed point.

\begin{definition}
A subgroup $H\in\SG$ is \emph{relatively amenable} if for any non-empty convex compact $G$-space there is a $H$-fixed point.
\end{definition}

Clearly, every amenable subgroup of $G$ is relatively amenable. We denote by $\RASG$ the closed subset of relatively amenable subgroups of $G$.

\begin{definition}\label{amenIRS}
An IRS $\mu$ of $G$ is \emph{relatively amenable} if $\mu(\RASG)=1$ and it is said to be \emph{amenable} if $\mu(\ASG)=1$.
\end{definition}


At the end of the introduction of \cite{Abert:2012jk}, the authors state that if $G$ is a linear group any amenable IRS lies in $\RG$ (see \cite{Gelander:2014kq} for a proof). They implicitly ask if the same holds in the general case. The same question also appeared in the introduction of \cite{Biringer:2014fr} and in \cite[\S7]{Tucker-Drob:2012vn}. We prove the following statement which yields a positive answer as a corollary.

\begin{theorem}\label{radrm}
Any relatively amenable IRS of $G$ lies in the amenable radical.
\end{theorem}

If $H\leq G$ is a closed subgroup of $G$, we identify $\mathcal{S}(H)$ with the closed subspace of $\SG$ consisting of closed subgroups of $G$ included in $H$. More precisely,  Theorem \ref{radrm} means that if $\mu\in\IRS$ and $\mu(\RASG)=1$ then $\mu(\mathcal{S}({\RG}))=1$.

\begin{corollary}\label{rad} Any amenable IRS of $G$ lies in the amenable radical of $G$.
\end{corollary}

\begin{remark} This theorem is related to \cite[Theorem 5.4]{MR2059438} which concerns strongly non-amenable groups, namely discrete groups with positive first Betti number.
\end{remark}

%

This previous theorem allows us to extend  \cite[Theorem 5]{Abert:2012jk} --- with the same proof --- outside the linear world. Let $\Gamma$ be a group generated by a finite symmetric set $S$. A sequence $(H_n)$ of finite index subgroups of $\Gamma$ is said to \emph{locally approximates} $\Gamma$ if the Schreier graphs $\mathrm{Sch}(\Gamma/H_n,S)$ converge to the Cayley graph $\mathrm{Cay}(\Gamma,S)$ in Benjamini-Schramm convergence \cite{MR1873300}. We obtain the following theorem, where $\rho_0\!\left(\mathrm{Sch}(\Gamma/H_n,S)\right)$ and $\rho(\mathrm{Cay}(\Gamma,S))$ are the spectral radii of the Markov averaging operator on respectively $\ell^2_0(\mathrm{Sch}(\Gamma/H_n,S))$ and $\ell^2(\mathrm{Cay}(\Gamma,S))$.

\begin{theorem}Let $\Gamma$ be a finitely generated  group with  trivial amenable radical and let $S$ be a finite symmetric generating set of $\Gamma$. Let $(H_n)$ be a sequence of subgroups of finite index such that $|\Gamma: H_n|\to\infty$ and
\[\overline{\lim}\ \rho_0\!\left(\mathrm{Sch}(\Gamma/H_n,S)\right)\leq\rho(\mathrm{Cay}(\Gamma,S)).\footnote{All along this text $\overline{\lim}$ and $\underline{\lim}$ denote respectively the limit superior and the limit inferior of  a sequence of real numbers.}\]
Then $(H_n)$ locally approximates $\Gamma$.
\end{theorem}

\begin{proof}Using Theorem \ref{rad} instead of \cite[Theorem 3]{Abert:2012jk} the proof of \cite[Theorem 5]{Abert:2012jk} can be repeated verbatim.
\end{proof}

As explained in \cite{Tucker-Drob:2012vn}, Theorem \ref{rad} also yields positive answers to \cite[Questions 7.2, 7.4 \& 7.5]{Tucker-Drob:2012vn}, see Figure 1 in that paper. Namely, arguments there and Theorem \ref{rad} show the following.

\begin{theorem}\label{shift minimal}
Let $\Gamma$ be a countable group with trivial amenable radical. Any non-trivial p.m.p action of $\Gamma$ that is weakly contained in the Bernoulli shift $\Gamma\action[0,1]^\Gamma$, is free.
\end{theorem}

\begin{remark} A group satisfying the conclusion of Theorem \ref{shift minimal} is said to be \emph{shift-minimal}. The work in \cite{Tucker-Drob:2012vn} leads to the conclusion that shift-minimality is equivalent to triviality of the amenable radical.
\end{remark}

\subsection{Kazhdan property (T)}
An opposite property to amenability is Kazhdan property (T). This property is hereditary for finite covolume subgroups in $G$ (see for example \cite[Theorem 1.7.1]{MR2415834}). So, it is natural to hope it is also hereditary for IRSs; at least when there is no strict closed subgroup containing the IRS. In that case, the IRS is said to be \emph{spanning} (see \S\ref{span}).

\begin{definition}\label{rt}Let $\mathcal{S}_{r(T)}(G)$ be the subset of groups $H\in\SG$ such that $(G,H)$ has relative property (T). An IRS $\mu$ is said to have relative property (T) if $\mathcal{S}_{r(T)}(G)$  has measure 1.
\end{definition}


\begin{theorem}\label{(T)}Assume $G$ is finitely generated. If $G$ has a spanning IRS with relative property (T) then $G$ has property (T).

\end{theorem}

\begin{remarks}
\begin{enumerate}
\item It is clear from the definitions that if $G$ has property (T) then any of its IRS has relative property (T).
\item Theorem \ref{(T)} does not hold if one removes the finite generation assumption. Consider the group $G=\oplus_{n\in\N}\Z/2\Z$. This countable group is not finitely generated and thus does not have property (T). Let $\delta_n$ be the Dirac measure at the $n$-th copy of $\Z/2\Z$ and let $\mu$ be $6/\pi^2\sum_{n\in\N}\delta_n/n^2$. The measure $\mu$ is a spanning IRS with relative property (T).
\item For a result toward compact generation see Proposition \ref{compactgeneration}.
\end{enumerate}
\end{remarks}
%
%
%
%
%
%
%
%
%

\noindent
\textbf{Structure of the paper.} In Section \ref{cone}, we study the cone of convex weak*-compact subsets of a dual Banach space. We show that the set of convex weak*-compact subsets in the dual unit ball is a convex compact space itself. This construction will allow us to construct the barycenter of a measure on convex compact sets. The proof of Theorem \ref{rad} appears in Section \ref{amenirs}. The last section is devoted to relative property (T).

\section{Locally convex structure on the cone of convex weak*-compact subspaces}\label{cone}
For all this section we fix some real separable Banach space $(E,\|\ \|)$, we define $E_1$ to be its unit ball and $E^*$ its dual Banach space with its unit ball $E^*_1$. Any topological statement for subsets of $E^*$ will be relative to the weak*-topology. We aim to define a locally convex (Hausdorff) topological vector space $\mathcal{E}$ in which one can embed the set $\CR$ of compact convex (non-empty) subspaces of $E^*$.

First observe that $\CR$ is an \emph{abstract cone} \cite[Chapter 3]{Cones} with the operations $A+B=\{a+b;\ a\in A,\ b\in B\}$ and $\lambda A=\{\lambda a;\ a\in A\}$ for $\lambda\geq 0$ --- those sets are also easily seen to be compact convex subsets. 
For $C\in\CR$ and $b\in E_1$, we set $b^+(C)=\max_{c\in C}b(c)$ and $b^-(C)=\min_{c\in C} b(c)$.


Let $\mathcal{E}$ be the vector space $\prod_{b\in E_1}\R_b$, with the product topology $\tau$. This is a locally convex (Hausdorff) topological vector space. By the Hahn-Banach Separation Theorem, we get an injective  map
\[\begin{array}{rcl}
f\colon\CR&\to&\mathcal{E}\\
C&\mapsto& (b^+(C))_{b\in E_1}.
\end{array}\]
Observe that $b^+(C+C')=b^+(C)+b^+(C')$, $b^+(\lambda C)$=$\lambda b^+(C)$ for $\lambda\geq0$ and $b^+(\{-c;\ c\in C)\}=-b^-(C)$. This means that the previous operations defined on $\CR$ are the same as the ones coming from the vector space in $\mathcal E$: the abstract cone $\CR$ can actually be realized as a cone in $\mathcal E$. In particular $\CR$ is a convex subspace of $\mathcal E$.

 We endow $\CR$ with the induced topology, that is the coarsest topology such that all linear forms $b^+$ are continuous. Furthermore if $G$ acts by isometries on $E$, it also acts by linear homeomorphisms on $\mathcal{C}$. Indeed for $g\in G$ and $b\in E_1$, the map $C\mapsto b^+(gC)$ is $(bg)^+$.

If $C\in\CR$, we define $\Conv(C)=\{C'\in\CR;\ C'\subseteq C\}$ and $\CR_1=\Conv(E^*_1) $.

\begin{lemma}\label{compactness} If $C'\subseteq C$, then $\Conv(C')$ is a closed convex subset of $\Conv(C)$ and $C$ is an extreme point of $\Conv(C)$. Moreover $\CR_1 $ is compact  and if $E$ is separable then $\CR_1$ is metrizable.
\end{lemma}

\begin{proof}The first part comes from the fact that $C'\subseteq C$ if and only if for any $b\in E_1$, $b^+(C')\leq b^+(C)$. If $C',C''\in\Conv(C)$ and $C=(C'+C'')/2$ then for any $b\in E_1$, $b^+(C)=(b^+(C')+b^+(C''))/2$ and since $b^+(C'),\ b^+(C'')\leq b^+(C)$, one has $b^+(C')=b^+(C'')=b^+(C)$. Thus $C'=C''=C$. This proves $C$ is an extreme point of $\Conv(C)$.

To prove compactness, first observe that $f(\CR_1)\subseteq \prod_{b\in E_1}[-1,1]$. Let $t=(t_b)$ be a point in the closure of $f(\CR_1)$. We will show there is $C\in\CR_1$ such that $f(C)=t$.


\emph{Claim}: For all finite subset $F\subset E_1$, there is $C_F\in\CR_1$ such that for all $b\in F$, $b^+(C_F)=t_b.$\\

Assume the claim holds true. In that case, one can moreover assume that
\[C_F=\bigcap_{b\in F}b^{-1}\left((-\infty,t_b]\right)\cap E^*_1.\]
 With this assumption, observe that $F\subseteq F'$ implies $C_{F'}\subseteq C_F$. Now define $C=\cap_{F\in\mathcal{F}(E_1)}C_F$ where $\mathcal{F}(E_1)$ denotes the  set (directed for reverse inclusion order) of finite subsets of $E_1$. Fix $b\in E_1$.
 For $F\in\mathcal{F}(E_1)$, choose $c_F\in C_F$ such that $b(c_F)=b^+(C_F)$. As $E_1^*$ is compact, the net $(c_F)$ has a convergent subnet with limit $c$. Since $c_{F'}\in C_F$ for $F\subseteq F'$, the point $c$ is in $C$ and continuity implies $b(c)=b(c_F)=t_b$ for $F$ containing $b$. Since $C\subseteq C_F$, one has $b^+(C)\leq b^+(C_F)=t_b$. Thus $b^+(C)=t_b$ for all $b\in E_1$ and $f(C)=t$.

It remains to show the claim. Fix $F\in\mathcal{F}(E_1)$. There is a sequence $(C^k)$ with $C^k\in\CR_1$ such for all $b\in F$, $b^+(C^k)\to t_b$. For each $b\in F$ choose $c^k_b\in C^k$ such that $b(c^k_b)=b^+(C^k)$. Up to extraction one may assume that $c^k_b$ converges to some $c_b$ as $k\to\infty$. Thus for $b,b'\in F$, $b(c_{b'})=\lim_{k\to\infty}b(c^k_{b'})\leq \lim_{k\to\infty}b^+(C^k)=t_{b}$ and $b(c_b)=t_b$. Define $C_F$ to be the closed convex hull of $\{c_b\}_{b\in F}$. By construction $C_F$ satisfies the conditions of the claim.

For metrizability, let $(b_n)$ be a dense countable subset of $E_1$ and assume $(C_\alpha)$ is a net in $\CR_1$ such that $b_n^+(C_\alpha)\to b_n^+(C)$ for any $n\in\N$. We have
\begin{align*}|b^+(C)-b^+(C_\alpha)|&\leq|b^+(C)-b^+_n(C)|+|b^+_n(C)-b^+_n(C_\alpha)|+|b^+_n(C_\alpha)-b^+(C_\alpha)|\\
&\leq 2\|b_n-b\|+|b^+_n(C)-b^+_n(C_\alpha)|
\end{align*}
which shows that $|b^+(C)-b^+(C_\alpha)|\to 0$. Thus, the embedding of $f\colon\CR_1\to\prod_{n\in\N}\R_{b_n}$ yields the same topology as $\tau$ on $\CR_1$. This shows metrizability.
\end{proof}

\begin{remark}The proof of Lemma \ref{compactness} shows that $\CR_1$ embeds as a convex bounded subspace in $\ell^\infty(E_1)$. Since weak*-topology and the topology of pointwise convergence coincide on bounded subsets of $\ell^\infty(E_1)$, $\CR_1$ can be seen as a weak*-compact convex subset of $\ell^\infty(E_1)$ (seen as the dual of $\ell^1(E_1)$).

If $G$ acts continuously by linear isometries on $E$, then there are obvious linear isometric adjoint actions of $G$ on $\ell^1(E_1)$ and $\ell^\infty(E_1)$ that are not continuous. However the restriction of this action to $\CR_1$ and to the weak*-closure of its span $L$ is continuous. In particular $\CR_1$ is a general\footnote{We use the adjective \emph{general} to emphasize that, contrarily to Zimmer's original definition, a general convex compact $G$-space is merely a convex compact $G$-invariant subset of some locally convex vector with a continuous affine action of $G$ (not necessarily in the dual of some separable Banach space on which $G$ acts).} convex compact $G$-space. Moreover $L$ can be identified with the dual of some Banach space $L^\flat$ which can be realized as the Banach space quotient of $\ell^1(E_1)$ by the intersection of the kernels of all elements of $L$.
 \end{remark}

 The following Lemma is a key step in the proof of Theorem \ref{radrm}:

 \begin{lemma}\label{prob(conv(C))2}
 Assume $E$ is separable, $G$ acts by isometries on $E$ and  $C\in\CR_1$ is $G$-invariant. If there is no invariant closed convex proper subspace of $C$ then the only $G$-invariant Borel probability measure on $\CR(C)$ is $\delta_C$.
 \end{lemma}

\begin{proof}
Let $\nu\in\Prob(\CR(C))^G$. Let $C_0$ be its barycenter, that is $C_0=\int C'\dd\nu(C')$. The integration process is the vector-valued integration in locally convex vector spaces (see e.g. \cite[Theorem 3.27]{MR1157815} or \cite[IV \S 7 N\textsuperscript{0} 2]{MR2018901}). Recall the barycenter is uniquely defined via the relation $\varphi\left(\int C'\dd\nu(C')\right)=\int \varphi(C')\dd\nu(C')$ for all $\varphi\in \mathcal{E}^*$. For $b^+\in E_1$ and $g\in G$,
\begin{equation*}b^+(gC_0)=b^+ g(C_0)=\int b^+ g(C)\dd\nu(C)=\int b^+(gC)\dd\nu(C)=\int b^+(C)\dd\nu(C)=b^+(C_0).\end{equation*}
Since $E_1$ separates points in $\Conv(C)$, $C_0$ is $G$-invariant and thus $C=C_0$ by minimality. Choose a dense countable set $(b_n)$ in $E_1$. Since for all $C'\subset C$ compact convex, $b_n^+(C')\leq b^+_n(C)$, the above equality implies that $\{b_n^+(C')=b^+_n(C);\ C'\in\CR(C)\}$ has measure one for all $n$, and hence $C'=C$ almost surely\footnote{One may also rely on the fact that $C$ is an extreme point and corollary \cite[IV \S 7 N\textsuperscript{0} 2]{MR2018901} tells us $\nu=\delta_C$.}. It follows that $\nu=\delta_C$.
\end{proof}

\section{Spanning IRSs}\label{span}The following lemma yields the existence of a minimal closed subgroup in which an IRS lies. The existence of such minimal subgroup was already proved in \cite{Hartman:2013yf}. This subgroup is called the \emph{normal closure} of the IRS. We include a proof for completeness.
\begin{lemma}\label{minsub}Let $\mu\in\IRS$. There exists a unique minimal closed subgroup $N\leq G$ such that $\mu(\mathcal{S}(N))=1$. This group is moreover normal.
\end{lemma}

\begin{proof}
Let $\mathcal{H}$ be the family of all closed subgroups $H$ such that $\mu(\mathcal{S}(H))=1$. We define $H_0$ to be $\bigcap_{H\in\mathcal{H}}H$.
Normality, closeness and uniqueness are immediate consequences of the definition, thus it suffices to prove that $\mu(\mathcal{S}(H_0))=1$.

Choose a countable subset $\{g_n\}_{n\in\N}$ of the open subset $G\setminus H_0$ such that $G\setminus H_0=\bigcup_{n\in\N} B(g_n,\rho_n)$ where $\rho_n=d(g_n,H_0)/2$ . For $n\in\N$, choose $H_{n}\in\mathcal{H}$ such that $H_{n}\cap \overline{B}(g_n,\rho_n)=\emptyset$. We have $H_0=\bigcap_{n\in\N}H_{n}$ and thus $\mathcal{S}\left(H_0\right)=\bigcap_{n\in\N}\mathcal{S}\left(H_{n}\right)$. The latter intersection being countable, we have $\mu(\mathcal{S}(H_0))=1$.
\end{proof}

\begin{definition}An IRS on $G$ is \emph{spanning} if $G$ is the normal closure of the IRS.\end{definition}

To illustrate this definition, one can reformulate \cite[Corollary 1.2]{Biringer:2014fr} in the following way: If $G$ has a spanning unimodular IRS then $G$ is unimodular itself. If $\mu\in\IRS$ is spanning, it is proved in  \cite{Hartman:2013yf} that for any $S\subseteq\SG$ with $\mu(S)=1$, $G=\overline{\langle\cup_{H\in S}H\rangle}$. Since an IRS of $G$ is also an IRS of its normal closure, considering spanning IRSs is a natural reduction to prove results on IRS. One has to be careful with relative properties since an IRS may have a property relatively to $G$ and this property may fails relatively to the normal closure (for example think to a normal subgroup with relative property (T) which does not have property (T)).

The following lemma is close to the \emph{locally essential lemma} \cite[Lemma 2.2]{Gelander:2014kq}.

\begin{lemma}\label{positiv}Assume $G$ is countable and let $\mu$ be an IRS. The normal closure $N$ of $\mu$ coincides with the subgroup generated by $\{h\in G;\ \mu\left(\{H\in\SG;\ h\in H\}\right)>0\}$.
\end{lemma}

\begin{proof} For $h\in G$, denote by $S_h=\{H\in\SG;\ h\in H\}$. Let $N_0$ be the subgroup of $G$ generated by $\{h\in G;\ \mu(S_h)>0\}$ and let $N$ be the normal closure of $\mu$. We aim to prove that $N=N_0$.  Let $S$ be the complement of $\cup_{g\in G\setminus N_0}S_g$ in $\SG$. We have $\mu(S)=1$ and $S=\mathcal{S}(N_0)$. In particular $N\leq N_0$. If $h\notin N$ then $S_h\cap \mathcal{S}(N)=\emptyset$ and $\mu(S_h)=0$. Thus, if $\mu(S_h)>0$ then $h\in N$. This shows that $N_0\leq N$.
\end{proof}

\section{Amenable IRSs}\label{amenirs}

Amenability has many equivalent definitions. We use the following one (which appears in \cite[4.1.4]{MR776417} for example). Let $H$ be a topological group. A \emph{convex compact H-space} $C$ is a $H$-invariant convex weak*-compact subspace of the unit ball of the dual of a separable Banach space on which $H$ acts continuously by isometries. A topological group $H$ is \emph{amenable} if every non-empty convex compact $H$-space contains a $H$-fixed point.

An important and open question about the space of amenable subgroups $\ASG$ is to decide whether it is closed in $\SG$. This question was investigated in \cite{Caprace:2013kq}, in which the authors decided to introduce a weaker notion of amenability for closed subgroups.  A subgroup $H\in\SG$ is \emph{relatively amenable} if for any non-empty convex compact $G$-space there is a $H$-fixed point.

It is easy to prove that the space of relatively amenable subgroups $\RASG$ is closed in $\SG$ \cite[Lemma 18]{Caprace:2013kq}. Of course, $\ASG\subseteq\RASG$ and it is an open question to decide whether it is actually an equality. The group $G$ is said to belong to the class $\mathscr{X}$ if there is equality. This class $\mathscr{X}$ is quite large since it contains, for example, discrete groups, connected groups, algebraic groups over local fields, groups amenable at infinity, and is stable under some natural extension processes \cite[ Theorem 2]{Caprace:2013kq}.

\begin{proof}[Proof of Theorem \ref{radrm}]
We fix a locally compact second countable group $G$ and an IRS $\mu$ satisfying $\mu(\RASG)=1$.
We let $N$ be the normal closure of $\mu$.
We will argue to show that $N$ is amenable, hence being normal, it is contained in the amenable radical of $G$.
In fact, we will show that $N$ is relatively amenable in $G$, and use the easy fact that normal subgroups are amenable iff they are relatively amenable (see e.g \cite[Proposition~3]{Caprace:2013kq}).
That is, we need to show that every convex compact $G$-space has an $N$-fixed point.
We fix such a $G$-space, $C$ --- a $G$-invariant convex weak*-compact subspace of the unit ball of a dual of a separable Banach space $E$,
on which $G$ acts continuously by isometries.
Without loss of generality (by Zorn lemma and a compactness argument) we assume as we may that $C$ has no proper $G$-invariant closed convex subset.
Let $K\lhd G$ be the kernel of the action on $C$.
We claim that for $\mu$-a.e $H\in \mathcal{G}$, $H<K$.
The proof of the theorem follows from the claim: by the definition of $N$, $\mu(\mathcal{S}(K))=1$ implies $N<K$ and $C$ is $N$-fixed.


\begin{lemma}\label{borelness} The map $H\mapsto \Fix(H)$ from $\RASG$ to $\Conv(C)$ is Borel and $G$-equivariant.
\end{lemma}

\begin{proof} Choose a countable dense subset $(b_n)$ of $E_1$ as in the proof of Lemma \ref{compactness}. Since the $b^+_n$ are countable and define the topology, it suffices to prove that for any $n\in\N$, $H\mapsto b^+_n(\Fix(H))$ is Borel. We actually prove that $H\mapsto b^+_n(\Fix(H))$ is upper semi-continuous. Fix a sequence $H_n$ converging to $H$ in $\RASG$ and choose $c_k\in\Fix(H_k)$ such that $b(c_k)=b^+_n(\Fix(H_k))$. Let $c$ be a limit point (up to extraction) of $(c_k)$. Since the action $G\action C$ is continuous, $c\in \Fix(H)$ and since $b_n$ is continuous, one has $b_n(c)=\lim_{k\to\infty} b_n(c_k)$. Thus $b_n^+(\Fix(H))\geq \overline{\lim}\ b^+_n(\Fix(H_k))$. The $G$-equivariance is clear.
\end{proof}

Denote the image of $\mu$ under $H\mapsto\Fix(H)$ by $\nu$. Clearly, $\nu$ is a $G$-invariant Borel probability measure on the compact metrizable space $\Conv(C)$. By Lemma \ref{prob(conv(C))2}, it follows that $\nu=\delta_C$, meaning that $\mu$-almost every $H\in\SG$ fixes every point of $C$, and the claim is proven.
\end{proof}


\section{Kazhdan Property (T)}\label{kazhdan}
For a unitary representation $\pi$ of $G$, we denote by $Z^1(G,\pi)$ the space of 1-cocycles, by $B^1(G,\pi)$ the space of coboundaries and by $\overline{H^1}(G,\pi)=Z^1(G,\pi)/\overline{B^1(G,\pi)}$ the first reduced cohomology associated to $\pi$. We refer to \cite{MR2415834} for standard facts about those objects. We recall the following important theorem due to Y. Shalom \cite[Theorem 6.1]{MR1767270}, see also \cite[Theorem 3.2.1]{MR2415834}.

\begin{theorem}Assume $G$ is compactly generated. The group $G$ has property (T) if and only if for any irreducible unitary representation $\pi$, $\overline{H^1}(G,\pi)=0$.\end{theorem}

We will use the following characterization of relative property (T) that holds for locally compact second countable groups.

\begin{proposition}[{\cite{MR2154620}}] Let $H\in\SG$. The pair $(G,H)$ has relative property (T) if and only if for any continuous isometric action of $G$ on a Hilbert space there are $H$-fixed points.
\end{proposition}

We observe that relative property (T) is actually a Borel property and thus  $\mathcal{S}_{r(T)}(G)$ is measurable for any IRS.

\begin{lemma} The subset $\mathcal{S}_{r(T)}(G)$ is a Borel subset of $\SG$.
\end{lemma}

\begin{proof} We use this quantitative characterization of relative property (T) from \cite{MR2154620}. The pair $(G,H)$ has relative property (T) if and only if for every $\delta>0$ there is pair $(Q,\varepsilon)$ (consisting of a compact subset and a positive number) such that for any unitary representation of $G$ with $(Q,\varepsilon)$-invariant unit vector $v$, there is a $H$-fixed invariant vector at distance less than $\delta$ from $v$.

For a pair $(Q,\varepsilon)$ define $\Psi{(Q,\varepsilon)}$ to be the set of functions of positive type $\psi$ (see \cite[Definition C.4.1]{MR2415834}) with $\psi(e)=1$ and $\inf_{g\in Q} \Re(\psi(g))\geq 1-\varepsilon/2$. For such a function the GNS construction yields a unitary representation with a $(Q,\varepsilon)$-invariant unit vector $v$ such that $\psi(g)=\langle gv,v\rangle$. Recall that a unitary representation of  a group  $H$ with a $\left(H,\sqrt{2}\right)$-invariant unit vector has a non-zero invariant vector \cite[Proposition 1.1.5]{MR2415834}. Conversely if $v$ is $(Q,\varepsilon)$-invariant vector in some unitary representation of $G$, then the function $\psi$ defined by $\psi(g)=\langle gv,v\rangle$ belongs to $\Psi(Q,\varepsilon)$. Now choose  $\alpha\in\left(0,1/\sqrt{2}\right)$. With the characterization and the reminder we have

\[\mathcal{S}_{r(T)}(G)=\bigcup_{(Q,\varepsilon)}\bigcap_{\psi\in\Psi{(Q,\varepsilon)}}\left\{H\in\SG;\ \inf_{g\in H}\Re(\psi(g))\geq1-\alpha\right\}.\]
To conclude, it suffices to observe that $\left\{H\in\SG;\ \inf_{g\in H}\Re(\psi(g))\geq1-\alpha\right\}$ is a closed subset and that the union can be replaced by a countable one thanks to $\sigma$-compactness.
\end{proof}

We start the proof of Theorem \ref{(T)} by dealing with finite-dimensional representations.
A topological group is said to have \emph{property (FE)} if any continuous isometric action on a Euclidean space has a fixed point. A subgroup $H<G$ has \emph{relative property (FE)} if for every continuous isometric action of $G$ on a Euclidean subspace, $H$ fixes a point. We mimic Definition \ref{rt} to define IRSs with relative property (FE).

\begin{proposition}\label{FEAR}If $G$ has a spanning IRS with relative property (FE) then $G$ has property (FE). \end{proposition}
The proof of this proposition relies on similar methods as the proof of the main theorem in \cite{DGLL}. We prove Proposition \ref{FEAR} without emphasizing questions about measurability of the constructions. The interested reader may have a look at  \cite{DGLL} for those questions.

Proposition \ref{FEAR} will follow  from the following lemma. If $E$ is a Euclidean space, we endow the space of all closed convex subspaces $\CR(E)$ with the coarsest topology such that $x\mapsto d(x,C)$ is a continuous function on $\CR(E)$ for every $x\in E$. This is the so-called Wijsman topology \cite{MR1269778}.

\begin{lemma}\label{prob(conv(E))2}
Let $E$ be a finite dimensional Euclidean space. Assume that $G$ acts by isometries on $E$, without fixed points, and irreducible linear part. The only $G$-invariant Borel probability measure on $\CR(E)$ is $\delta_E$.
 \end{lemma}

\begin{proof}
Let $\nu\in\Prob(\CR(E))^G$.  Fix $x_0\in E $ and look at the function $f:E\to\R$ defined by

\[f(x)=\int_{\CR(E)}d(x,C)-d(x_0,C))\dd\nu(C).\]

This continuous function satisfies the cocycle relation $f(gx)=f(x)+f(gx_0)$.
Either $f$ achieves a minimum or not.
By \cite[Lemma 2.4]{MR1645958}, if $f$ has no minimum then  $G$ fixes a point at infinity of $E$. In that case the linear part of $G$ stabilizes the direction of this point. As we assumed the representation to be irreducible, this cannot be the case.

Thus $f$ has a minimal set $M$, which is a $G$-invariant and convex. If this convex subset is bounded, then it has a $G$-invariant circumcenter, contradicting the assumption. If it is unbounded, then by \cite[Lemma 1.7]{MR1645958} and the fact that $G$ does not stabilize a point at infinity, we conclude that $M$ is an affine subspace of $E$. By the irreducibility assumption, it follows that $M=E$, hence that $f$ is constant.

 In particular, one has that $x\mapsto d(x,C)$ is affine for almost all $C\in\CR(E)$. Since the distance to a strict convex subspace is not affine (it is non constant and does not take negative values), one has $C=E$ almost surely, hence $\nu=\delta_C$.
\end{proof}

\begin{proof}[Proof of Proposition \ref{FEAR}] Let $E$ be a Euclidean space with an action of $G$. Up to consider a minimal invariant affine subspace we assume there is no invariant affine subspace of $E$. Let $\mu\in\IRS$ with relative property (FE).

First we consider the case when the linear part of the action of $E$ is irreducible. In that case, by assumption, almost every $H\in\SG$ fixes an affine subspace of $E$. Pushing forward $\mu$ by the map $H\mapsto \Fix(H)$, we get a $G$-invariant measure on $\CR(E)$. By Lemma \ref{prob(conv(E))2}, this measure is $\delta_E$, meaning that almost every $H\in \SG$ fixes $E$ pointwise. Since the IRS is spanning, this means that $G$ fixes $E$ pointwise.

In the general case, any orthogonal representation can be written as the orthogonal sum of irreducible representations $(\pi_1,E_1),\dots,(\pi_n,E_n)$. If $b$ is the cocycle associated to the action of $G$ on $E$ and $P_i:E\to E_i$ is the projection then $P_i\circ b$ is a cocycle for the representation $\pi_i$ on $E_i$. Hence by the previous case, this associated affine action on $E_i$ fixes a point for every $i$. This implies that $G$ fixes a point in $E$.
 \end{proof}

In the  proof of Theorem \ref{(T)}, we will use the notion of weakly mixing unitary representations. Recall that these are unitary representations without finite dimensional subrepresentation. Moreover the tensor product of two weakly mixing representations is still weakly mixing.

\begin{lemma}\label{weak-mixing} Let $(\pi,\mathcal{H})$ be a separable weakly mixing unitary representation of $G$. 
Then the only $G$-invariant measure on $\mathcal{H}$ is $\delta_{\{0\}}$.
\end{lemma}

\begin{proof}
Let $\nu$ be such a measure and $(\overline{\pi},\overline{\mathcal{H}})$ the conjugate representation \cite[Definition A.1.10]{MR2415834}. Consider the map $\Phi:\mathcal H\to \mathcal{H}\otimes\overline{\mathcal{H}}$ defined by $\Phi(x)=x\otimes x/\|x\|^2$ for $x\neq0$ and $\Phi(0)=0$.  Now, $\int_{\mathcal H}\Phi(x)\dd\mu(x)$ is a $G$-invariant vector of $\mathcal{H}\otimes\overline{\mathcal{H}}$ which has to be 0 because of weak-mixing. In particular, for any $v\in\mathcal{H}$ and almost all $x\in \mathcal H$, we have $|\langle v,x\rangle|^2=\langle v\otimes v,x\otimes x\rangle=0$. Since $\mathcal{H}$ is separable, $x=0$ for almost all $x\in \mathcal H$, meaning that $\nu=\delta_{\{ 0\}}$

\end{proof}

\begin{proof}[Proof of Theorem \ref{(T)}]We consider a spanning IRS $\mu$ with relative property (T). Let $\pi$ be an irreducible unitary representation of $G$ with underlying Hilbert space $\mathcal{H}$ (which we may assume separable since $G$ is second countable), $b\in Z^1(G,\pi)$ and consider the associated affine action $\alpha$ given by $\alpha(g)v=\pi(g)v+b(g)$ for any $v\in\mathcal{H}$. The case where $\mathcal{H}$ has finite dimension is treated in Proposition \ref{FEAR}. From now on, we assume $\mathcal{H}$ has infinite dimension, thus weakly mixing.

For any $H\in\mathcal{S}_{r(T)}(G)$, $\Fix(H)$ is a non-empty affine subspace of $\mathcal{H}$. For $H_1,\dots,H_n\in\mathcal{S}_{r(T)}(G)$ and $\varepsilon>0$, we define the set of $(\cup_iH_i,\varepsilon)$-fixed points
\[F(H_1,\dots,H_n,\varepsilon)=\{v\in\mathcal{H};\ \forall h\in H_1\cup\dots\cup H_N,\ ||\alpha(h)v-v||<\varepsilon\}.\]
We claim there is $A_n\subseteq\SG^n$, with $\mu^n(A_n)=1$ such that   for all $\varepsilon>0$ and $(H_1,\dots,H_n)\in A_n$, $F(H_1,\dots,H_n,\varepsilon)\neq \{0\}$. We prove it by induction. The case $n=1$ follows from $\mu(\mathcal{S}_{r(T)}(G))=1$. Assume this is true for $n-1$, let $v(H_1,\dots,H_{n-1};H_n)$ be the vector of minimal norm in the closed convex set $ \overline{F(H_1,\dots,H_{n-1},\varepsilon/3)-\Fix(H_n)}$. The map $(H_1,\dots,H_n)\mapsto v(H_1,\dots,H_{n-1};H_n)$ satisfies the assumption of Lemma \ref{weak-mixing} and thus is 0. In particular, there is some $v$ which is at distance less than $\varepsilon/3$ from both $F(H_1,\dots,H_{n-1},\varepsilon/3)$ and $\Fix(H_n)$ and thus $v\in F(H_1,\dots,H_n,\varepsilon)$.

Let $F=\{g_1,\dots,g_n\}\subseteq G$ be a finite subset, thanks to Lemma \ref{positiv} one can find $(h_i^j)$ such that $g_i=h_i^1\cdots h_i^{k(i)}$ and $\mu\left(S_{h_i^j}\right)>0$.

Let $N=\sum_{i=1}^nk(i)$. The product set $P=\prod_{i=1}^n (S_{h_i^1}\times\cdots\times S_{h_i^{k(i)}})\subseteq \SG^N$ has also positive measure and thus intersects $A_{N}$. In particular, one can find $\left(H_1^1,\dots,H_n^{k(n)}\right)$ in $P$ such that $F(H_1^1,\dots,H_n^{k(n)},\varepsilon/N)\neq\emptyset$. Now, observe that $g_i\in H_i^1\cdots H_i^{k(i)}$ and thus taking  $v\in F(H_1^{1},\dots,H^{{k(n)}}_n,\varepsilon/N)$, we get a $(F,\varepsilon)$-invariant vector. Thus $\overline{H^1}(G,\pi)=0$  and Shalom's theorem concludes the proof.
\end{proof}

\begin{remark}In case one considers non spanning IRSs, our methods lead to the following statement for  $G$ countable: The normal closure of an IRS with relative property (T) has relative property $\barFH$. We refer to \cite{MR2892915} for the definition of property $\barFH$ and its relative version. We emphasize that even if Shalom's theorem shows that property (T) and property $\barFH$ coincide for compactly generated groups, the relative versions do \emph{not} coincide.
\end{remark}

The following proposition is an adaptation of the classical result of compact generation for locally compact groups with property (T) (see for example \cite[Theorem 1.3.1]{MR2415834}).

\begin{proposition}\label{compactgeneration}If $G$ has a spanning ergodic IRS with relative property (T) then $G$ is compactly generated as a normal subgroup.
\end{proposition}

\begin{proof}
Let $\mathcal{C}$ be the set of open  compactly generated subgroups of $G$. For $F\in\mathcal{C}$ denote by $\ell^2(G/F)$ the quasi-regular representation of $G$. We consider the Hilbertian sum
\[\mathcal{H}=\bigoplus_{F\in\mathcal{C}}\ell^2(G/F).\]

From construction, the diagonal representation of $G$ on $\mathcal{H}$ almost has  invariant vectors and thus for almost all $H\in\SG$, $H$ has non-trivial invariant vectors. In particular, for such $H$, there is $F\in\mathcal{C}$ such that $H$ has a non-trivial invariant vector $f\in\ell^2(G/F)$. Let $g\in G$ such that $f(gF)\neq0$. There are $h_1,\dots,h_n$ such that $HgF\subseteq h_1gF\cup\dots\cup h_ngF$ that is $H\subseteq h_1F^g\cup\dots\cup h_nF^g$. In particular, $H$ is contained in at most $n$ right classes of the normal closure $\langle F\rangle_G$ of $F$ and thus there is $F'\in\mathcal{C}$ (namely, $F'=\langle F^g,h_1,\dots,h_n\rangle$) such that $H\leq F'$ and thus $H\leq\langle F'\rangle_G$. This last condition is a $G$-invariant closed condition, thanks to ergodicity, there is $F'\in\mathcal{C}$ such that it holds for almost all $H\in\SG$. That is almost surely $H\leq\langle F'\rangle_G$. Finally the spanning property implies that $G=\langle F'\rangle_G$.
\end{proof}

\begin{remark} One may ask if the normal closure of an IRS with relative property (T) is actually compactly generated. Observe that in the proof of Proposition \ref{compactgeneration} we did not use the general assumption that $G$ is second countable. We give a counterexample for a locally compact group which is not second countable.

Consider the lamplighter group over the circle group $G=\Z/2\Z\wr S^1$ with the topology coming from the discrete topology on $\bigoplus_{S^1}\Z/2\Z$ and the usual one on $S^1$. For $s\in S^1$ let $\delta_s$ be the Dirac measure on $\SG$ at the $\Z/2\Z$ copy at coordinate $s$. Let $\mu=\int_{S^1}\delta_s\dd s$. This is an ergodic IRS with normal closure $\bigoplus_{S^1}\Z/2\Z$ which is generated by one element as normal subgroup but is not finitely generated.\end{remark}

\appendix
\setcounter{proposition}{0}
\renewcommand{\thesection}{A}
\section{The set of amenable closed subgroups}\label{appendix}
\begin{center}Phillip Wesolek \end{center}
\medskip
The class of l.c.s.c. groups $G$ such that the set $\mathcal{S}_a(G)\subseteq \mathcal{S}(G)$ is closed in the Chabauty topology is very large by work of Caprace and Monod \cite[Theorem 2]{Caprace:2013kq}. Caprace and Monod raise the following question: \textit{Is $\mathcal{S}_a(G)$ closed in $\mathcal{S}(G)$ for all l.c.s.c. groups $G$?} It is natural to consider a weaker question: \textit{Is $\mathcal{S}_a(G)$ a Borel set in $\mathcal{S}(G)$?} We here answer the latter question in the affirmative.

\begin{theorem}\label{thm:A}
For every l.c.s.c. group $G$, $\mathcal{S}_a(G)$ is a Borel set in $\mathcal{S}(G)$. 
\end{theorem}

\begin{ack} 
The author thanks U. Bader, B. Duchesne, and J. L\'{e}cureux for including this appendix in their work and for their many helpful comments and suggestions. 
\end{ack}

\subsection{The class $\ms{G}$ and its closure properties}
We consider the class of all l.c.s.c. groups $G$ such that $\mathcal{S}_a(G)$ is a Borel set in $\mathcal{S}(G)$; we denote this class by $\ms{G}$.

\begin{theorem}\label{thm:A_properties}
The class $\ms{G}$ enjoys the following permanence properties:
\begin{enumerate}[(1)]

\item If $N\trianglelefteq G$ is an amenable closed subgroup and $G/N\in\ms{G}$, then $G\in \ms{G}$.

\item If $H\leq G$ is a finite index closed subgroup of $G$ and $H\in \ms{G}$, then $G\in \ms{G}$.

\item If $G_0,G_1\in \ms{G}$, then $G_0\times G_1\in \ms{G}$.

\end{enumerate}
\end{theorem}

\begin{proof}
For $(1)$, let $\pi:G\rightarrow G/N$ be the usual projection and define $\Pi:\mathcal{S}(G)\rightarrow \mathcal{S}(G/N)$ by $C\mapsto \ol{\pi(C)}$. The Borel sigma algebra of $\mathcal{S}(G/N)$ is generated by sets of the form $O_U:=\{C\mid C\cap U= \emptyset\}$ where $U$ ranges over open subsets of $G/N$. To verify $\Pi$ is Borel measurable, it is therefore enough to check $\Pi^{-1}(O_U)$ is a Borel set for each open $U\subseteq G/N$. This, however, is immediate since $\Pi^{-1}(O_U)=O_{\pi^{-1}(U)}$. \par

\indent To prove $(1)$, it now suffices to show $\mathcal{S}_a(G)=\Pi^{-1}(\mathcal{S}_a(G/N))$. For the forward direction, every compact, convex $\ol{\pi(C)}$-space is a compact convex $C$-space via the map $\pi$. We conclude there is an $C$-fixed point and, therefore, a $\ol{\pi(C)}$-fixed point. Via the fixed point criterion, $\ol{\pi(C)}$ is amenable. Conversely, suppose $\Pi(C)\in \mathcal{S}_a(G/N)$. So $\pi^{-1}(\Pi(C))\in \mathcal{S}(G)$, and $\pi^{-1}(\Pi(C))/N$ is amenable. Since amenability is stable under group extension, $\pi^{-1}(\Pi(C))\in \mathcal{S}_a(G)$. It now follows that $C\in \mathcal{S}_a(G)$. \par

\medskip

Claim $(2)$ is immediate. Indeed, the map $\Phi:\mathcal{S}(G)\rightarrow \mathcal{S}(H)$ via $C\mapsto C\cap H$ is Borel with $\Phi^{-1}(\mathcal{S}_a(H))=\mathcal{S}_a(G)$.\par

\medskip

\indent To see $(3)$, let $\pi_i$ for $i\in \{0,1\}$ be the projection onto the $i$-th coordinate. As with $(1)$, these maps induce Borel measurable maps $\Pi_i:\mathcal{S}(G_0\times G_1) \rightarrow \mathcal{S}(G_i)$. We claim 
\[
\mathcal{S}_a(G_0\times G_1)=\Pi_0^{-1}(\mc{S}_a(G_0))\cap \Pi_1^{-1}(\mc{S}_a(G_1)),
\]
from which the result follows. The forward direction follows as with $(1)$. Conversely, suppose $H\in \Pi_0^{-1}(\mc{S}_a(G_0))\cap \Pi_1^{-1}(\mc{S}_a(G_1))$, so $\ol{\pi_0(H)}\times \ol{\pi_1(H)}$ is amenable. Since $H\leq \ol{\pi_0(H)}\times \ol{\pi_1(H)}$ is a closed subgroup, we conclude that $H$ is amenable.\end{proof}

\subsection{T.d.l.c.s.c. groups}

We now show all totally disconnected locally compact second countable (t.d.l.c.s.c.) groups lie in $\ms{G}$. To do so, we will use F\o lner's condition \cite[Theorem G.5.1]{MR2415834}. A priori, F\o lner's condition is non-Borel since there is quantification over uncountable sets. We give a restatement that eliminates this problem. Our restatement makes use of an old theorem of D. van Dantzig: \textit{A t.d.l.c. group admits a basis at the identity of compact open subgroups}. See, for example, \cite[(7.7)]{MR551496}.

\begin{proposition}\label{prop:borel_folner_td} 
Let $G$ be a t.d.l.c.s.c. group with left invariant Haar measure $\mu$. Suppose $\mc{F}:=(g_i)_{i\in \N}$ is a countable dense subset of $G$ and $(V_i)_{i\in \N}$ is an $\subseteq$-decreasing basis at $1$ of compact open subgroups. Then the following are equivalent:
\begin{enumerate}[(1)]
\item $G$ is amenable.

\item (Borel F\o lner's condition) For all finite non-empty $F\subseteq \mc{F}$ and for all $n\geq 1$, there is a finite non-empty $H\subseteq \mc{F}$ and $i\in \N$ such that for $U:=\bigcup_{h\in H} h V_{i}$
\[
\frac{\mu(g_iU\Delta U)}{\mu(U)}\leq \frac{1}{n}
\]
for all $g_i\in Q:=\bigcup_{f\in F}fV_0$. 
\end{enumerate}
\end{proposition}

\begin{proof}
For the reverse direction, let $K\subseteq G$ be compact. Since $\mc{F}$ is dense, there is a finite $F\subseteq \mc{F}$ such that $K\subseteq Q:=\bigcup_{f\in F}fV_0$. Fixing $n\geq 1$, condition $(2)$ now supplies a $U=\bigcup_{h\in H} h V_{i}$ for some finite non-empty $H\subseteq \mc{F}$ that satisfies the inequality condition for a dense subset of $Q$. It suffices to show the inequality holds for \textit{all} $x\in Q$. To this end, fix $x\in Q$ and let $g_j\rightarrow x$ be such that the inequality holds for all $j$. By taking a sufficiently large $k$, we have $h^{-1}g^{-1}_kxh\in V_i$ for all $h\in H$. Hence, $g_kU=xU$, and the inequality holds for $x$. \par

\medskip

\indent Conversely, suppose $Q:=\bigcup_{f\in F}fV_0$ with $F\subseteq \mc{F}$ finite and non-empty. Fix $n>0$ and $0<\delta<\frac{1}{n}$ and apply the F\o lner condition to find a Borel $U$ with $0<\mu(U)<\infty$ such that 
\[
\frac{\mu(xU\Delta U)}{\mu(U)}\leq \delta <\frac{1}{n}
\]
for all $x\in Q$. \par

\indent We now approximate $U$ by a set of the desired form. Fix $\epsilon>0$ small enough so that
\[
\frac{1}{1-\epsilon} \left(\frac{\mu(xU\Delta U)}{\mu(U)} +4\epsilon \right)<\frac{1}{n}
\]
for all $x\in Q$. By inner and outer regularity of the Haar measure, we may find a compact, non-empty $K$ and open $O$ such that $K\subseteq U \subseteq O$, $\mu(K)\geq (1-\epsilon)\mu(U)$, and $\mu(O)\leq (1+\epsilon)\mu(U)$. Now for each $k\in K$, there is $V_j$ and $g\in \mc{F}$ with $k\in gV_j\subseteq O$. The $gV_j$ form an open cover of $K$, hence there is some finite, non-empty $H'\subseteq \mc{F}$ so that $K\subseteq \Omega:=\bigcup_{h\in H'} h V_{j(h)}$. By taking $i=\max\{j(h)\mid h\in H'\}$ and possibly expanding $H'$ by finitely many elements of $\mc{F}$, $\Omega=\bigcup_{h\in H}hV_i$. So $\Omega$ has the correct form.\par

\indent We here argue $\Omega$ satisfies $(2)$ for $Q$ and $\frac{1}{n}$. By construction, $K\subseteq \Omega\subseteq O$, whereby
\[
(1-\epsilon)\mu(U)\leq \mu(\Omega)\leq (1+\epsilon)\mu(U).
\]
Thus, $\frac{1}{\mu(\Omega)}\leq\frac{1}{(1-\epsilon)\mu(U)}$. On the other hand, recall $\mu(A\Delta B)=:d_{\mu}(A,B)$ gives a pseudometric on the algebra of measurable sets that is invariant under measure preserving maps; consider \cite[Chapter 1]{MR2583950}. For any $x\in Q$, we therefore have
\[
\begin{array}{ccl}
d_{\mu}(x\Omega, \Omega)-d_{\mu}(xU, U)  & \leq & d_{\mu}(x\Omega,xU)+d_{\mu}(xU,U)+d_{\mu}(U,\Omega)-d_{\mu}(xU,U)\\
										& = & 2d_{\mu}(\Omega,U)\\
										& \leq & 2\mu(O\setminus \Omega)+2\mu(O\setminus U)\\
										&\leq & 4\epsilon \mu(U)
\end{array}
\] 
and $\mu(x\Omega \Delta \Omega)\leq \mu(xU\Delta U )+4\epsilon \mu(U)$. We conclude
\[
\frac{\mu(x\Omega \Delta \Omega)}{\mu(\Omega)}\leq \frac{1}{1-\epsilon} \left(\frac{\mu(xU\Delta U)}{\mu(U)} +4\epsilon \right)<\frac{1}{n},
\]
hence $\Omega$ satisfies $(2)$ for $Q$ and $\frac{1}{n}$.
\end{proof}

The next lemma gives a technique for computing the Haar measure in a Borel way.

\begin{lemma}\label{lem:haar}
Suppose $G$ is a t.d.l.c.s.c. group with left invariant Haar measure $\mu$ and let $(V_i)_{i\in \N}$ be an $\subseteq$-decreasing basis at $1$ of compact open subgroups. For all non-empty compact open sets $O$ and $L$, there are finite sets $W\subseteq O$ and $K\subseteq L$ and $i\in \N$ such that $O=\bigsqcup_{w\in W}wV_i$ and $L=\bigsqcup_{k\in K}kV_i$. Therefore, $\frac{\mu(O)}{\mu(L)}= \frac{|W|}{|K|}$.
\end{lemma}

\begin{proof}
For each $o\in O$ there is $V_{i(o)}\in (V_i)_{i\in \N}$ such that $oV_{i(o)}\subseteq O$, so $O=\bigcup_{o\in O}oV_{i(o)}$. Since $O$ is compact, there is a finite, non-empty set $W\subseteq O$ such that $O=\bigcup_{o\in W} oV_{i(o)}$. We may likewise write $L=\bigcup_{l\in K}lV_{i(l)}$ for $K\subseteq L$ finite and non-empty. Taking $i=\max\{i(x)\mid x\in W\cup K\}$ and possibly expanding $W$ and $K$ by finitely many elements, we have $O=\bigcup_{o \in W} oV_i$ and $L=\bigcup_{l\in K}lV_i$. We eliminate redundant cosets to conclude $O=\bigsqcup_{o \in W'} oV_i$ and $L=\bigsqcup_{l\in K'}lV_i$ for some finite, non-empty $W'\subseteq O$ and $K'\subseteq L$. 
\end{proof}

\begin{theorem}\label{thm:tdlcsc} 
If $G$ is a t.d.l.c.s.c. group, then $\mathcal{S}_a(G)$ is a Borel set in $\mathcal{S}(G)$.
\end{theorem}

\begin{proof}
Fix $(V_i)_{i\in \N}$ an $\subseteq$-decreasing basis at $1$ of compact open subgroups for $G$. Let $(d_i)_{i\in \N}$ be a set of Kuratowski$-$Ryll-Nardzewski selector functions for $\mathcal{S}(G)$. That is to say, a set of Borel functions $d_i:\mathcal{S}(G)\rightarrow G$ such that $\{d_i(C)\}_{i\in \N}$ is dense in $C$ for all $C\in \mathcal{S}(G)$; see \cite[(12.13)]{MR1321597} for example. \par

\indent For each $C\in \mathcal{S}(G)$, $(C\cap V_i)_{i\in \N}$ forms an $\subseteq$-decreasing basis at $1$ of compact open subgroups of $C$, and $(d_i(C))_{i\in \N}$ forms a countable dense subset. In view of Proposition~\rm\ref{prop:borel_folner_td}, $C\in \mathcal{S}_a(G)$ if and only if for all finite non-empty $F\subseteq \N$ and for all $n\geq 1$, there is a finite non-empty $H\subseteq \N$ and $i\in \N$ such that for $U:=\bigcup_{h\in H} d_h(C) \left(C\cap V_{i}\right)$
\[
\frac{\mu_C(d_j(C)U\Delta U)}{\mu_C(U)}\leq \frac{1}{n}
\]
for all $d_j(C)\in Q:=\bigcup_{f\in F}d_f(C)\left( C\cap V_0\right)$ where $\mu_C$ is the left invariant Haar measure on $C$. For $F\subseteq \N$ finite non-empty, $n>0$, $j\in \N$, $H\subseteq \N$ finite non-empty, and $i\in \N$, put
\[
\Omega:=\left\{C\in \mathcal{S}(G)\mid \frac{\mu_C(d_j(C)U\Delta U)}{\mu_C(U)}\leq \frac{1}{n} \text{ with }U=\bigcup_{h\in H} d_h(C) \left(C\cap V_{i}\right)\right\}
\]
and
\[
\Sigma:=\left\{C\in \mathcal{S}(G)\mid d_j(C)\notin Q \right\}.
\]
Since $\mc{S}_a(G)$ is a combination of countable intersections and unions of sets with the forms of $\Omega$ and $\Sigma$, showing $\Omega$ and $\Sigma$ are Borel proves the theorem. \par

\indent The case of $\Sigma$ is immediate: For each $f\in F$, define $\Phi_f:\mc{S}(G)\rightarrow G$ by $C\mapsto d_f(C)^{-1}d_j(C)$. This is a Borel measurable map since the functions $d_i$ are Borel and the group operations are continuous. We thus see that
\[
\Sigma =\bigcap_{f\in F}\Phi_f^{-1}(V^c_0)
\]
and, therefore, is Borel.\par

\indent To see $\Omega$ is Borel, we apply Lemma~\rm\ref{lem:haar}. Indeed, $C\in \Omega$ if and only if either $d_j(C)U=U$ or there is $l\geq i$ and finite non-empty sets $W,K\subseteq \N$ with $\frac{|W|}{|K|}\leq \frac{1}{n}$ so that
\begin{enumerate}[(1)]
\item $xU\Delta U=\bigsqcup_{w\in W}d_w(C)(C\cap V_l)$ and

\item $U= \bigsqcup_{k\in K}d_k(C)(C\cap V_l)$.
\end{enumerate}
It is easy to verify the latter statement is Borel. We check the first disjunct: $d_j(C)U=U$ if and only if
\begin{enumerate}[(1)]
\item $\forall h\in H\;\forall m\in \N\; \exists k\in H$ so that $d_k(C)^{-1}d_j(C)d_h(C)d_m(C\cap V_i)\in V_i$ and 

\item $\forall h\in H\;\forall m\in \N\; \exists k\in H$ so that $d_k(C)^{-1}d_j(C)^{-1}d_h(C)d_m(C\cap V_i)\in V_i$.
\end{enumerate}
Let $\Phi_{h,m,k}:\mathcal{S}(G)\rightarrow G$ by $C\mapsto d_k(C)^{-1}d_j(C)d_h(C)d_m(C\cap V_i)$ and $\Psi_{h,m,k}:\mathcal{S}(G)\rightarrow G$ by $C\mapsto d_k(C)^{-1}d_j(C)^{-1}d_h(C)d_m(C\cap V_i)$. Since the $d_i$ are Borel functions and the group operations are continuous, $\Phi_{h,m,k}$ and $\Psi_{h,m,k}$ are Borel, hence 
\[
\{C\in \mathcal{S}(G)\mid d_j(C)U=U\}=\bigcap_{h\in H}\bigcap_{m\in \N}\bigcup_{k\in H}\Phi_{h,m,k}^{-1}(V_i)\cap \bigcap_{h\in H}\bigcap_{m\in \N}\bigcup_{k\in H} \Psi_{h,m,k}^{-1}(V_i)
\] 
is Borel. The second disjunct follows similarly. We conclude $\mathcal{S}_a(G)$ is a Borel subset of $\mathcal{S}(G)$ verifying the theorem.
\end{proof}

\begin{proof}[Proof of Theorem~\ref{thm:A}]
Let $G$ be a l.c.s.c. group and form $R_a(G)$, the amenable radical of $G$. In view of Theorem~\ref{thm:A_properties}, if $G/R_a(G)\in \ms{G}$, then $G\in \ms{G}$, so we may assume $R_a(G)=\{1\}$ without loss of generality. Via \cite[Theorem 11.3.4]{MR1840942}, there is $H\leq G$ a finite index closed subgroup that is a direct product of a connected group and a totally disconnected group. Since $H\in \ms{G}$ implies $G\in \ms{G}$, we may reduce again to $G\simeq G_0\times G_1$ with $G_0$ a connected l.c.s.c. group and $G_1$ a t.d.l.c.s.c. group. Now \cite[Theorem 2]{Caprace:2013kq} gives that $G_0\in \ms{G}$, and $G_1\in \ms{G}$ via Theorem~\ref{thm:tdlcsc}. We conclude $G\in \ms{G}$ verifying the theorem.

\end{proof}

\bibliographystyle{halpha}
\bibliography{biblio.bib}

\end{document}